\documentclass[reqno]{amsart}

\usepackage[centering]{geometry}

\usepackage{graphicx}

\usepackage{mathtools}
\usepackage{mathrsfs}
\usepackage{amsmath}
\usepackage{bm}
\usepackage{tikz-cd}
\usepackage{amsthm}
\usepackage{enumitem}

\definecolor{calpolypomonagreen}{rgb}{0.12, 0.3, 0.17}
	\definecolor{asparagus}{rgb}{0.53, 0.66, 0.42}
	\definecolor{applegreen}{rgb}{0.55, 0.71, 0.0}
		\definecolor{darkpastelgreen}{rgb}{0.20, 0.70, 0.24}
		\definecolor{amaranth}{rgb}{0.85, 0.17, 0.31}
			\definecolor{darkpastelred}{rgb}{0.76, 0.23, 0.13}
\usepackage[colorlinks=true, linkcolor = darkpastelred, citecolor = darkpastelgreen, urlcolor = ]{hyperref}

\newtheorem{theorem}{Theorem}[section]
\newtheorem{lemma}[theorem]{Lemma}

\theoremstyle{definition}

\newtheorem{definition}[theorem]{Definition}
\newtheorem*{definition*}{Definition}

\theoremstyle{definition}

\usepackage{amssymb}
\renewcommand{\leq}{\leqslant}
\renewcommand{\geq}{\geqslant}

\newcommand{\ca}{\mathcal}

\usepackage{todonotes}

\def\G{\mathrm{G}}

\def\G{\mathbf G}

\def\F{\mathbb F}
\def\Q{\mathbb Q}

\def\Zint{\mathbb Z}

\def\ca{\mathcal}

\def\t{\mathbf t}
\def\x{\mathbf x}
\def\n{\mathbf n}
\def\a{\mathbf a}
\def\c{\mathbf c}
\def\y{\mathbf y}

\def\m{\mathbf m}
\def\b{\mathbf b}
\def\g{\mathbf g}

\newcommand{\gen}[1]{\ensuremath{\langle #1\rangle}}

\makeatletter
\def\blfootnote{\xdef\@thefnmark{}\@footnotetext}
\makeatother

\author[L. Bary-Soroker]{Lior Bary-Soroker}
\address{School of Mathematical Sciences, Tel Aviv University, Tel Aviv 69978, Israel }
\email{barylior@tauex.tau.ac.il}
\author[D. Garzoni]{Daniele Garzoni} 
\address{Department of Mathematics, University of Southern California, Los Angeles, CA 90089-2532, USA}
\email{garzoni@usc.edu}
\author[V. Matei]{Vlad Matei}
\address{Institute of Mathematics “Simion Stoilow" of the Romanian Academy, Calea Grivitei 21, Bucharest 010702, Romania}
\email{vlad.matei@imar.ro}

\begin{document}

\title{On the irreducibility of $f(2^n,3^m,X)$ and other such polynomials}

\maketitle

\begin{abstract}
	Let $f(t_1, \ldots, t_r, X)\in \Zint[t_1, \ldots, t_r,X]$ be irreducible and let $a_1, \ldots, a_r\in \Zint \smallsetminus \{0,\pm 1\}$. Under a necessary ramification assumption on $f$, and conditionally on the Generalized Riemann Hypothesis, we show that for almost all integers $n_1, \ldots, n_r$, the polynomial $f(a_1^{n_1}, \ldots, a_r^{n_r}, X)$ is irreducible in $\Q[X]$.
\end{abstract}

\section{Introduction and main result}
Hilbert's Irreducibility Theorem (HIT) is one of the central theorems in arithmetic geometry. In a quantitative form, it says that if $f(t_1, \ldots, t_r,X) \in \Zint[t_1,\dots, t_r,X]$ is an irreducible polynomial, then 
\begin{equation}\label{classicalHIT}
\frac{\#\{(n_1, \ldots, n_r)\in (\Zint\cap [-N,N])^r \mid f(n_1, \ldots, n_r,X)\in \Q[X] \text{ is irreducible} \}}{(2N+1)^r} \to 1,
\end{equation}
as $N \to \infty$ (see for example, \cite[Theorem 3.4.4]{Serre_Topics}). 
One may view $f$ as a cover of the algebraic group $G=\G_a^r$. 
In recent years, there is an extensive study of generalizations of the theorem to other algebraic groups $G$ \cite{corvaja2017zannier_hilbert,corvaja2022demeio_lombardo,bary2023ramified,bary2023garzoni}. 

A key observation of Zannier \cite{zannier2010} is to restrict to ramified covers, cf.\ \cite{corvaja2017zannier_hilbert}. For this Zannier introduces the so-called Pull Back (PB) condition. In our setting, the (PB) condition may be stated in polynomial terms, as follows.

We denote $r$-tuples with bold letters, e.g., $\t=(t_1,\ldots, t_r)$, and if $\n$ is an $r$-tuple of integers, we write  $\t^\n = (t_1^{n_1},\ldots, t_r^{n_r})$.
Let $f(\t,X)\in \Zint[\t,X]$ be a polynomial in $r+1$ variables and let $\a\in (\Zint\smallsetminus\{0,\pm 1\})^r$. 
We say that $f$ satisfies the (PB) condition, defined in \cite{zannier2010}, cf.\ \cite{debes1992irreducibility}, if: 
\begin{itemize}
    \item[(PB)] For every $\m\in (\Zint_{>0})^r$, the polynomial $f(\t^{\m}, X)$ is irreducible in $\overline{\Q}[\t,X]$, and $\deg_Xf\geq 1$.
\end{itemize}
In terms of ramification, (PB) is equivalent to $C=\{ f=0\}\subseteq \G_m^r\times \mathbb{A}^1$ being geometrically irreducible and the cover $\widetilde C\to \G_m^r$ having no unramified nontrivial subcovers, where $\widetilde C$ denotes the normalization of $C$. 

Denote by  
\begin{equation}
    \label{def:coun}
    \ca N(f,\a; N) = \#\{ \n\in ([-N,N]\cap \Zint)^r | f(\a^{\n},X)\in \Q[X] \text{ is irreducible} \} .
\end{equation}
In particular, \cite[Theorem~1]{zannier2010} implies that under (PB)
there is at least a positive density of $\n$ that keep irreducibility; that is to say,
\begin{equation}
	\label{eq:zannier}
    \liminf_{N\to \infty} 
\frac{\ca N(f,\a; N) }{(2N+1)^r} > 0.
\end{equation}
(In fact, \cite[Theorem~1]{zannier2010} is stated only for cyclic subgroups of $\G_m^r(\Q)$, which is the most difficult case; but his methods apply to our setting and give \eqref{eq:zannier}.)

Our main result shows that the density is $1$ under the Generalized Riemann Hypothesis (GRH).

\begin{theorem}\label{t:main}
    Let $f(\t,X)\in \ca \Zint[\t,X]$ satisfy \emph{(PB)} and let $\a\in ( \Zint \smallsetminus \{0,\pm 1\})^r$. Then, conditionally on GRH,
\[
    \lim_{N\to \infty}\frac{\ca N(f,\a; N)}{(2N+1)^r} = 1.
\]
\end{theorem}
The (PB) condition is necessary, as the polynomial $f(t,X) = X^2-t$ exemplifies. It does not satisfy (PB) and for, say, $a=2$, we have $\{n: f(2^n,X) \mbox{ is irreducible}\} =\{ n\equiv 1\mod 2\}$, hence the density is $1/2$. By considering $f(t,X)=X^2-2t$, one also notes that, in the definition of (PB), irreducibility in $\overline \Q[\t,X]$ cannot be relaxed to irreducibility in $\Q[\t,X]$.

For curves, that is when $r=1$, D\`ebes \cite{debes1992irreducibility} proves the  theorem unconditionally, obtaining more strongly the irreducibility of $f(a^n,X)$ for all but finitely many $n\in \Zint$. D\`ebes applies Siegel's theorem on integral points on curves (specifically, for the ramified sub-covers of the cover of $\G_m$ given by $f$). Since Siegel's theorem is restricted to curves, it seems that this approach cannot be applied in higher dimensions.

\subsection*{Acknowledgment}
The authors thank Noam Kimmel for a few helpful discussions. 

The first author was supported by the Israel Science Foundation (grant no. 702/19). The third author was supported by the project “Group schemes, root systems, and
related representations” founded by the European Union - NextGenerationEU through
Romania’s National Recovery and Resilience Plan (PNRR) call no. PNRR-III-C9-2023-
I8, Project CF159/31.07.2023, and coordinated by the Ministry of Research, Innovation
and Digitalization (MCID) of Romania.
\section{Method of Proof}
\label{sec:method}
We first prove the following theorem on rational points for general number fields: Let $K$ be a number field with ring of integers $\ca O_K$. For $f(\t,X)\in \ca O_K[\t,X]$ and $\a\in (\ca O_K)^r$ with $a_i$ nonzero and not roots of unity, we define
\begin{equation}
 \label{def:noroots}
    \ca N^0_K(f,\a; N) = \#\{ \n\in ([-N,N]\cap \Zint)^r | f(\a^\n,X) \text{ has no root in $K$} \} .
\end{equation}
The (PB) condition trivially generalizes to number fields, see Definition~\ref{def:PB}. 

\begin{theorem}
\label{t:main_no_roots}
    Assume the setting above and that $f$ satisfies (PB). Then, conditionally on GRH,  
    \[
       \lim_{N\to \infty} \frac{\ca N^0_K(f,\a; N)}{(2N+1)^r} = 1.
    \]
\end{theorem}

On the one hand, it is obvious that $\ca N^0_{\Q}(f,\a; N) \geq \mathcal{N}(f,\a; N)$, when $K=\Q$. However, this is not helpful. The key point is that given $f\in \Zint[\t,X]$ satisfying (PB), there is a number field $K$ and polynomials $f_i$ over $K$ satisfying (PB) and $\b\in (\ca O_K)^r$ with $b_i$ nonzero and not roots of unity, such that Theorem \ref{t:main_no_roots} for $\ca N^0_{K}(f_i,\b; N)$ implies Theorem \ref{t:main} for $\mathcal{N}(f,\a; N)$. So in other words the proof of the theorem over $\mathbb{Q}$ necessitates considering general number fields.

 We skip the details, as the deduction Theorem~\ref{t:main} from Theorem~\ref{t:main_no_roots}  is standard, and is essentially the same as the deduction of \cite[Theorem~1]{zannier2010} from \cite[Corollary]{zannier2010}. 

This approach automatically gives the generalization of Theorem~\ref{t:main} to number fields. To state the result,  we introduce the notation  $\mathcal{N}_K$, that is the obvious generalization of $\mathcal{N}$, when we replace $\mathbb{Q}$ and $\mathbb{Z}$ by $K$ and $\ca O_K$, respectively. 
\begin{theorem}
\label{t:main_number_fields}
Let $K$ be a number field, let $f(\t,X)\in \ca O_K[\t,X]$ satisfy \emph{(PB)} and let $\a\in (\ca O_K)^r$ be such that $a_i$ are nonzero and not roots of unity. Then, conditionally on GRH, 
\[
\lim_{N\to \infty}
\frac{\mathcal{N}_K(f,\a;N)}{(2N+1)^r} =1.
\]
\end{theorem}
Again, the deduction of Theorem~\ref{t:main_number_fields} from Theorem~\ref{t:main_no_roots} is standard, and we omit it.

Now we discuss the proof of Theorem~\ref{t:main_no_roots}. We would like to apply a standard reduction-modulo-primes method. 
The first step is to use that for a set of primes $p$ of density $1$, there is equidistribution modulo $p$, and hence one may bound the density of the complement of $\ca N^0$ modulo $p$
by $c<1$, using Chebotarev's theorem. The second step uses that reduction modulo several primes $p_1,\ldots, p_m$ is asymptotically independent, hence the density of the complement of $\ca N_0$ can be bounded by approximately $c^m$, which is very small if $m$ is large.

In our case, the first step necessitates that $\a^\n$ equidistributes in $\G_m^r(\F_p)$ for a set of primes $p$ of density $1$.  This is too much to expect to hold:
When $r=1$, $a^n$ equidistributes if and only if $a$ is a primitive root modulo $p$. It is open whether there are infinitely many such primes. Artin's primitive root conjecture (which is known to follow from GRH) predicts a positive density of primes for which $a$  is a primitive root (for $a\neq 0,\pm 1, \square$) and that this density is $<1$. So even Artin's conjecture is not sufficient.

For the second step, one needs that the events modulo different primes are asymptotically independent.  In our setting, the values $\a^\n \mod p$ and $\a^\n\mod q$ depend on $\n \mod p-1$ and $\n\mod q-1$, respectively. In the classical case, $p,q$ are coprime, but here $p-1,q-1$ are never coprime. To summarize, the following two points prevent the direct application of the classical method:
\begin{enumerate}
    \item The density of primes for which $\a^n$ equidistributes modulo $p$ is $<1$.
    \item For odd primes $p$ and $q$ we have $(p-1,q-1)>1$, so $\a^\n\mod p$ and $\a^\n\mod q$ are not independent.
\end{enumerate}

To overcome these problems, we modify the classical reduction-modulo-primes approach, so that it will be more flexible. For the first problem, we use the (PB) condition to relax the demand that $a_i^{n_i}$ are equidistributed in $\G_m(\F_p)$: it is sufficient that they are equidistributed in a large subgroup, see Lemma~\ref{l_one_prime}. 

To obtain equidistribution in a large subgroup for a large set of primes we use the following results: 
Let $K$ be a number field  and let $a\in \ca O_K$ be nonzero and not a root of unity. For $\ell>0$, let $d_\ell$ be the lower density of  the set of
primes $\mathfrak p$ of $K$ such that $\mathrm N_{K/\Q}(\mathfrak p)=p$ is prime and the order of $a$ modulo $\mathfrak p$ is at least $(p-1)/\ell$. In this setting, Erd\H{o}s-Murty 
\cite{erdos_murty} for $\Q$ and J\"arviniemi \cite{jarviniemi2021orders} for general number fields,  prove that GRH implies 
\begin{equation}
\label{l_GRH}
    \lim_{\ell\to \infty} d_\ell = 1,
\end{equation}
cf.\ \cite{Hooley}.

For the the second problem, we show that, in order to get asymptotic independence, it suffices to have many primes as above with the property that $(p-1,q-1)$ is  small for $p\neq q$, see Lemma~\ref{l_many_primes}. 
To estimate the number of such primes, we apply Tur\'an's theorem \cite{Turan}: Let $V$ be a finite simple undirected graph with $n$ vertices. If the number of edges of $V$ is at least $\delta n^2/2$, then $V$ contains a complete subgraph $K_r$ with 
\begin{equation}
    \label{l:turan}
    r\geq \frac{1}{1-\delta}.
\end{equation} 
The vertices of the graph will be primes $\mathfrak p$ from a specific set $\ca P_{f,\ell}(\a)$ of positive density (defined in \eqref{eq_pfell}) and with norm $p\in (x,2x]$, and we connect two primes $\mathfrak p,\mathfrak q$ iff $(p-1,q-1)$ is small; see Lemma~\ref{l:graph_theory} for details.

\section*{Notation List}
\begin{list}{}{\settowidth{\labelwidth}{00.00.0000.0}
    \setlength{\leftmargin}{\labelwidth}
    \addtolength{\leftmargin}{\labelsep}
    \renewcommand{\makelabel}[1]{#1\hfil}}
\item[$\bf a$] $(a_1,\ldots, a_r)\in \ca O_K^r$ such that each $a_i$ is nonzero and not a root of unity. 
\item[$\bf a^{\bf n}$, ${\bf a}^n$]  $(a_1^{n_1},\ldots, a_{r}^{n_r})$, $(a_1^n,\ldots, a_r^n)$, respectively.
 \item[$(a,b)$] the greatest common divisor of $a,b\in \Zint$.
\item[$\ca A_{\mathfrak p, M}$] the set of $\n\in (\Zint /M\Zint)^r$ such that $f(\a^\n,X)$ has a root modulo $\mathfrak p$ and $g_d(\a^\n)\neq 0$ modulo $\mathfrak p$, $\mathfrak p\in \ca P_f$ and $N_{K/\Q}(\mathfrak p)-1\mid M$ see \eqref{eq:Anp}.
\item[$d$] $\deg_X(f)$.
\item[$f({\bf t},X)$] a polynomial in $\ca O_K[{\bf t},X]$ satisfying (PB).
\item[$f_{\m}(\t,X)$] the polynomial $f(\t^\m,X)$.
\item[$f\gg_{a,b,\ldots} g$] $\exists C=C(a,b,\ldots)>0$ such that $|f(x)|\geq C|g(x)|$.  
\item[$f=o(g)$] $\displaystyle\lim_{x\to \infty} \frac{f(x)}{g(x)}=0$.
\item[$g_{d}({\bf t})$] the coefficient of $X^d$ in $f$.
\item[$\G_m$] the multiplicative group.
\item[$\ca O_K$] the ring of integers of $K$.
\item[$K$] a number field.
\item[$L_g$]
the algebraic closure of $K$ inside the Galois closure of $g\in \ca O_K[{\bf t},X]$ over $K({\bf t})$.
\item[$\overline{K}$] the algebraic closure of $K$.
\item[$\bf m$, $\bf n$] $(m_1,\ldots, m_r), (n_1,\ldots,n_r) \in \Zint^r$.
\item[{$[\m]$, $[m]$}]
the isogenies $\G_m^r\to \G_m^r$  given by ${\bf g}\mapsto {\bf g}^{\bf m}$, ${\bf g}\mapsto {\bf g}^{ m}$, respectively.
 \item[$\ca N(f,{\bf a}; N)$] the number of ${\bf n}$ with $|n_i|\leq N$ such that $f({\bf a},X)$ is irreducible over $\mathbb{Q}$, see \eqref{def:coun}.
 \item[$\ca N_K(f,{\bf a}; N)$] the number of $\bf n$ with $|n_i|\leq N$ such that $f({\bf a},X)$ is irreducible over $K$.
 \item[$\ca N^0_K(f,{\bf a}; N)$] the number of $\bf n$ with $|n_i|\leq N$ such that $f({\bf a},X)$ has no root in $K$, see \eqref{def:noroots}.
 \item[$\mathrm{N}_{K/\Q}(\mathfrak p)$] $|\ca O_K/\mathfrak p|$, the absolute norm of $\mathfrak p$.
  \item[$\mathfrak p$] a prime ideal of $\ca O_K$.
  \item[$\ca P_f$] the set of primes of $\ca O_K$ defined in \eqref{eq_p}.
  \item[$\ca P_{f,\ell}({\bf a})$] the subset of primes of $\ca P_f$ defined in \eqref{eq_pfell}.
 \item[$\bf t$]  $(t_1,\ldots, t_r)$, an $r$-tuple of independent variables. 
 \item[$Z_{C,N,\mathfrak p}$] $\{\n \in (\Zint \cap [-N,N])^r \mid \a^\n\in C(\F_p)\}$, where $C$ is a proper Zariski closed in $\G_m^r$.
 \item[$\delta(\ca Q)$] the density  
 $\displaystyle\lim_{x\to \infty}
 \frac{
 \#\{\mathfrak p\in \ca Q \mid x< \mathrm{N}_{K/\Q}(\mathfrak p)\leq 2x\}
 }{x/\log x}$ of a set of primes $\ca Q$ of $K$.
 \item[$\underline{\delta}(\ca Q)$] the lower density  
 $ \displaystyle\liminf_{x\to \infty}
 \frac{
 \#\{\mathfrak p\in \ca Q \mid  x<\mathrm{N}_{K/\Q}(\mathfrak p)\leq 2x\}
 }{x/\log x}$ of a set of primes $\ca Q$ of $K$.
\end{list}

\section{Preliminary lemmas}
\label{sec:preliminaries}

Let $K$ be a number field with a ring of integers $\ca O_K$.
\begin{definition} \label{def:PB}
A polynomial $f(\t,X)\in \ca O_K[\t,X]$ satisfies (PB) if $d:=\deg_X(f)\geq 1$, and for every $\m\in (\Zint_{>0})^r$, the polynomial $f_{\m}(\t,X):=f({\bf t}^{\bf m}, X)$ is irreducible in $\overline{K}[{\bf t},X]$.    
\end{definition}
For the rest of the section, fix $f(\t,X)\in \ca O_K[\t,X]$ satisfying (PB), let $d=\deg_X(f)$ and $g_d(\t)$ the coefficient of $X^d$.
For a polynomial $g(\t,X)\in\ca O_{K}[\t,X]$, let us denote by $L_g$ the algebraic closure of $K$ in the Galois closure of $f$ over $K(\t)$. 

\begin{lemma}
\label{l_field_scalars}
For every  $\m \in (\Zint_{>0})^r$,  $L_{f_{\m}}=L_f$.
\end{lemma}

\begin{proof}
Write $G=\G_m^r$ and $L=L_f$ for ease of notation, and let $\y = \t^m$ be regarded as an $r$-tuples of variables. Let $V\subseteq G\times A^{1}$ be the zero set of $f(\y,X)$, and let $\pi\colon V\to G; (\y,x)\mapsto \y$. Let $W$ be the Galois closure of $V\to G$ (namely, the normalization of $G$ in the Galois closure of the field extension $K(V)/K(G)$), and let $W'$ be the maximal unramified subcover of $W\to G$. In particular, the morphism $W'\to G$ factors through $W'\to W''\to G$, where $W''\cong G_L$ is the maximal scalar subcover, and $W'$ and $W$ can be regarded as geometrically integral $L$-varieties. Consider now the morphism $[\m]\colon G \to G$; $\g \mapsto \g^\m$; we base change along this morphism and get a diagram 
\[
 \begin{tikzcd}
W\times_{G,[\m]} G \arrow{r} \arrow{d} & W'\times_{G,[\m]} G \arrow{d} \arrow{r} & W''\times_{G,[\m]} G \arrow{d} \arrow{r} & G \arrow{d}{[\m]} \\
W \arrow{r} & W' \arrow{r} & W'' \arrow{r} & G.
\end{tikzcd}
\]
The Galois closure of the morphism $\{f(\t^m,x)=0\}\to G$; $(\t,x) \mapsto \t$ can be identified with an irreducible component of $W\times_{G,[\m]} G$. Note that $W''\cong W''\times_{G,[\m]} G\cong G_L$. This implies at once that $L \subseteq L_{f_\m}$. In particular, in order to conclude the proof it suffices to show that there exists $n$ so that $L = L_{f_{n\m}}$, where $n\m = (nm_1, \ldots, nm_r)$. Indeed, from this it follows that $L_{f_{n\m}} =L \subseteq L_{f_\m}\subseteq L_{f_{n\m}}$, whence equality holds throughout.

Choose $n$ so that the isogeny $G_L\to G_L$; $\g \to \g^{n\m}$ factors through $G_L\to W' \to G_L$. It follows that $W'\times_{G,[n\m]} G$ is isomorphic to a disjoint union of copies of $G_L$. In particular, letting $Y$ be an irreducible component of $W\times_{G,[n\m]} G$, we have that $Y \to G$ factors through $Y\to Y' \to W''\times_{G,[n\m]} G \to G$, where the middle map is an isomorphism. Since $Y\to Y'$ has no unramified subcovers, we deduce that $W''\times_{G,[n\m]} G$ is the maximal scalar subcover of $Y\to G$, which shows that $L_{f_{n\m}} =L$, as wanted.
\end{proof}

Let
\begin{equation}
\label{eq_p}
\ca P_f =\{\mathfrak p \mid  \mathfrak p \mbox{ is a prime of $\ca O_K$ satisfying (i)-(iii)}\},
\end{equation}
where
\begin{itemize}
	\item[(i)] $\mathrm N_{K/\Q}(\mathfrak p)=p$ is prime, so $\ca O_K/\mathfrak p \cong \F_p$.
    \item[(ii)] $f(\t,X) \in \F_p[\t,X]$ is  separable in $X$, $g_d(\t)\not\in \mathfrak p$, and $f(\t^d,X)$ is irreducible in $\overline{\F_p}[\t,X]$. 
    \item[(iii)] $\mathfrak p$ splits completely in $L_f$. 
\end{itemize}

(In (ii), we abuse notation and denote by $f\in \F_p[\t,X]$ the reduction of $f$ modulo $\mathfrak p$. We will freely adopt this convention from now on.) 
Recall that the density of a set of primes $\ca Q$ is defined by 
\[
\displaystyle\delta(\ca Q) = \lim_{x\to \infty}
 \frac{
 \#\{\mathfrak p\in \ca Q \mid  x<\mathrm{N}_{K/\Q}(\mathfrak p)\leq 2x\}
 }{x/\log x}.
\]
Similarly, we define the lower density $\underline{
\delta}(\ca Q)$ by replacing $\lim$ by $\liminf$.
\begin{lemma}\label{l deltaP>0}
    We have 
    $\delta(\ca P_f) = \frac{1}{[L_f:K]}>0$.
\end{lemma}

\begin{proof}
    The set of primes satisfying (i) has density $1$. Only finitely many primes divide the coefficients of $g_d$, or the coefficients of the discriminant of $f$. By \cite[Proposition 5.3, p. 241]{lang2013fundamentals}, there are only finitely many primes  for which $f(\t^d,X)$ is not irreducible in $\overline{\F_p}[\t,X]$, hence the primes satisfying (ii) also have density $1$. Finally, by Chebotarev's density theorem, the primes satisfying (iii) have density $1/[L_f:K]$. This finishes the proof.
\end{proof}

The following lemma follows from \cite[Proposition 2.1]{zannier2010}.

\begin{lemma}
\label{l:absolutely_irreducible}
    If $\mathfrak p\in \ca P_f$, then for every $\m \in (\Zint_{>0})^r$, $f(\t^\m, X) \in \F_p[\t,X]$ is separable, of degree $d$ in $X$, and irreducible in $\overline{\F_p}[\t,X]$.
\end{lemma}

\begin{proof}
Let $g_d(\t)$ and $\Delta(\t)$ be the leading coefficient and discriminant of $f$ as a polynomial in $X$. Then, $g_d(\t^\m)$ and $\Delta(\t^\m)$ are the leading coefficient and discriminant of $f(\t^\m,X)$. By condition (ii), $g_d(\t)$ and $\Delta(\t)$ are non zero modulo $\mathfrak p$, hence also $g_d(\t^m)$ and $\Delta(\t^\m)$.

Let $[\m]$ be the isogeny $(\G_m^r)_{\F_p} \to (\G_m^r)_{\F_p}$; $\g \mapsto \g^\m$, and  $[m]:=[(m,\ldots, m)]$, where $m$ is a positive integer.
If $\m=(m, \ldots, m)$ is a constant vector and $p\nmid m$, then the proof of \cite[Proposition 2.1]{zannier2010} applies here.
    
    If $\m=(m_1, \ldots, m_r)$ and $p\nmid m_1\cdots m_r$, then the isogeny $[m_1\cdots m_r]$ factors through $[\m]$, whence the statement follows from the previous paragraph. 
    
    Finally, if $p\mid m_1\cdots m_r$, let $[\m]=[\m']\circ [\m'']$ where for every $i=1,\ldots, r$, $m'_i$ is a $p$-power and $p\nmid m''_i$. By the previous paragraph $f(\t^{\m},X)$ is irreducible in $\overline{\F_p}[\t^{\m'},X]$. Moreover, $f(\t^{\m},X)$ is separable, while the cover $[\m']$ is purely inseparable, hence the corresponding extensions of $\overline{\F_p}(\t^{\m'})$ are linearly disjoint and so $f(\t^{\m},X)$ is irreducible in $\overline{\F_p}[\t,X]$.
\end{proof}

For a positive integer $\ell$ and $\a\in \ca O_K^r$ with $a_i$ nonzero and not a root of unity, let 
\begin{equation}
\label{eq_pfell}
\ca P_{f,\ell}(\a) \subseteq \ca P_f
\end{equation}
be the subset of primes $\mathfrak p\in \ca P_f$ such that  $a_1\cdots a_r\not\in \mathfrak p$, and the orders of $a_1, \ldots, a_r$ in $(\F_p)^\times$ are all at least $(p-1)/\ell$. 

\begin{lemma}
\label{l:large_ell}
    Assume GRH. If $\ell$ is sufficiently large depending on $f$ and $\a$, then $\underline \delta(\ca P_{f,\ell}(\a))>0$.
\end{lemma}

\begin{proof}
We have $\ca P_{f,\ell}(\a) = \ca P_{f} \cap \bigcap_{i=1}^r \ca P_i$, where $\ca P_i$ is the set of primes $\mathfrak p$ for which the order of $a_i$ modulo $\mathfrak p$ is $\geq (p-1)/\ell$. By Lemma~\ref{l deltaP>0}, $\delta(\ca P_f)>0$. By \eqref{l_GRH}, if $\ell$ is sufficiently large, then
$\underline \delta(\ca P_i)\geq 1-\frac{\delta(\ca P_f)}{2r}$. Hence, by a union bound,
\[
    \underline \delta(\ca P_{f,\ell}(\a))\geq 
    \delta(\ca P_{f}) -\sum_{i=1}^r  \frac{\delta(\ca P_f)}{2r} = \frac{\delta(\ca P_f)}{2}>0,
\]
as needed.
\end{proof}

The next lemma follows from applying Tur\'an's theorem (see \eqref{l:turan}) to a graph whose vertices are elements of a set of primes $\ca P$ of positive lower density.

\begin{lemma}\label{l:graph_theory}
    Let $\mathcal{P}$ be a set of primes of a number field $K$ of positive lower density, let $C>0$, let $x$  
    be sufficiently large depending on $\mathcal{P}$,  
    and let $0<z \le  \log(x)^C$. Then, there exist pairwise distinct primes $\mathfrak p_1, \ldots, \mathfrak p_t\in \ca P
    $, with $t\gg_{C,\mathcal{P}} 
    z$, $p_i:={\mathrm{N}}_{K/\Q}(\mathfrak p_i) \in (x,2x]$, $i=1,\ldots, t$, and $(p_i-1,p_j-1)\leq z$, for all $i\neq j$.
\end{lemma}

\begin{proof}
    In this proof we write $\gg$ for $\gg_{\mathcal{P}}$. 
    Let $V$ be the graph whose vertices are $\mathfrak p \in \ca P
    $ with $p:=\mathrm{N}_{K/\Q}(\mathfrak p) \in (x,2x]$. We connect $\mathfrak p\neq \mathfrak q$ by an edge if and only if $(p-1,q-1) 
    \leq z$. 
    Let $n\gg x/\log x$ be the number of vertices.

    Let $M$ be the number of pairs of vertices not joined by an edges, that is $M=\binom{n}{2}-e$, where $e$ is the number of edges.
    If $\mathfrak p\neq \mathfrak q$ are not connected by an edge, then there exists $d>z$ such that $d\mid(p-1)$ and $d\mid (q-1)$. The number of $\mathfrak p$ with fixed norm $p$ is $\ll 1$, hence 
    \[
        M\ll \sum_{d>z} A_d,
    \]
    where $A_d = \# \{ (p,q) : p\equiv 1 \pmod d, \ q\equiv 1\pmod d, \mbox{ and } x<p,q\leq 2x\}$. 
    
    First assume that $d\le z(\log x)^2$. Then, by the Siegel--Walfisz theorem, the number of primes $p\equiv 1\pmod d$ in $(x,2x]$ is $\ll_C x/(\phi(d)\log x)$. So $
    A_d \ll_C \frac{x^2}{(\log x)^2}\phi(d)^2$. 
    Since\footnote{By \cite[Theorem 2.14]{montgomery2007multiplicative}, $\sum_{d\leq x} d^2/\phi(d)^2=O(x)$. Apply summation by parts (Abel's summation formula) to $\sum \frac{1}{\phi(d)^2} = \sum \frac{d^2}{\phi(d)^2}\frac{1}{d^2}$ to get the desired inequality.}
    $\sum_{d>z} \frac{1}{\phi(d)^2} \ll \frac{1}{z}$, we conclude that 
    \[
        \sum_{z<d\le z(\log x)^2} A_d \ll_C \frac{x^2}{z (\log x)^2}.
    \]
    Next assume that $z(\log x)^2<d<2x$. Then, trivially $A_d\leq (2x)^2/d^{2}$, and so $\sum_{z(\log x)^2<d<2x}A_d\leq 4x^2\sum_{d>z(\log x)^2} d^{-2} \ll \frac{x^{2}}{z(\log x)^2}$. So 
    \begin{equation}\label{eq:boundM}
        M \ll \sum_{z<d\le z(\log x)^2} A_d  + \sum_{z(\log x)^2<d<2x} A_d \ll_C \frac{x^2}{z(\log x)^2} \ll_C \frac{n^2}{z}.
    \end{equation}
    Thus, the number of edges is 
    \[
        e = \binom{n}{2} - M  \geq \frac{n^2}{2}\Big(\frac{n-1}{n}-\frac{2C'}{z}\Big) = \frac{n^2}{2}\Big(1-\frac{C''}{z}\Big),
    \]
    where $C'>0$ is the implied constant given in \eqref{eq:boundM}. 
   We may assume $z\geq 2C''$, otherwise we simply take $t=1$. We deduce by \eqref{l:turan} that $V$ contains a complete subgraph with $t \gg_C z$ edges, which concludes the proof.
\end{proof}

\section{Reduction Lemmas}
Classically, there are two basic types of thin sets in the context of Hilbert's irreducibility theorem: A thin set of type I is the set of rational points in a proper Zariski closed subvariety. A thin set of  type II is the set of rational points which may be lifted to a rational point in a degree $\geq 2$ cover,  cf.\ \cite{Serre_Topics}. 
In this section we establish bounds for basic thin sets modulo primes (assuming (PB)).

\subsection{Thin set of type II}\label{subsc:oneprime}
Recall that $f=g_d(\t)X^d + \cdots \in \ca O_K[\t,X]$ is a polynomial satisfying (PB). 
Fix a prime $\mathfrak p\in \ca P_{f}$, let $p=\mathrm N_{K/\Q}(\mathfrak p)$, and let $M$ be a multiple of $p-1$. 
Let 
\begin{equation}\label{eq:Anp}
    \ca A_{\mathfrak p}=\ca A_{\mathfrak p,M}:=\{\n \in (\Zint/M\Zint)^r \mid f(\a^\n,X) \mbox{ has a root modulo $\mathfrak p$ and }  g_d(\a^\n)\neq 0\mod  \mathfrak p\}.
\end{equation}
Since $\a$ is considered modulo $p$ and $M$ is a multiple of $p-1$, then $\a^\n$ is well defined. 

\begin{lemma}
\label{l_one_prime}
Let $\ell \ge 1$ and let $\mathfrak p\in \ca P_{f,\ell}(\a)$ and $p=\mathrm N_{K/\Q}(\mathfrak p)$. Let $\n$ be a random variable taking the values of $(\Zint/M\Zint)^r$ uniformly. Then,
    \[
        {\rm Prob}(\n\in  \ca A_{\mathfrak p}) \leq 1 -\frac{1}{d} + O_{f,\ell}( p^{-1/2}).
    \]
\end{lemma}
\begin{proof}
    In this proof, all 
    polynomials are considered to be over $\F_p$; for brevity, we somtimes omit this from the notation. We will abbreviate and write $a_i$ also for the image of $a_i$ in $\F_p^\times$. Since the value of $\a^\n$ is defined by $\n\mod (p-1)$ and since the pushforward of the uniform measure is also uniform, we may assume without loss of generality that $M=p-1$. 

    For every $i$, we may write $a_i=b_i^{m_i}$, where $\gen{b_i}=\F_p^\times$, $m_i\le \ell$ and $(p-1)/m_i$ is the order of $a_i$. Let $f_\m(\t,X) = f(\t^\m,X)$. By Lemma~\ref{l:absolutely_irreducible}, $f_\m$ is separable in $X$, irreducible in $\overline{\F_p}[\t,X]$, and $\deg_X(f_\m)=\deg_X(f)=d$. In particular, the leading coefficient $g_{d,\m}(\t) := g_{d}(\t^\m)$ of $f_\m$ is nonzero.
        By Lemma~\ref{l_field_scalars}, $L_f=L_{f_\m}$, so $\mathfrak p$ totally splits in $L_{f_\m}$. We get 
    \[
    \begin{split}
        \ca A_{\mathfrak p} &= \{ \n \in (\Zint/(p-1)\Zint)^r \mid f_{\m}(\b^\n,X) \mbox{ has a root modulo $\mathfrak p$ and }  g_{d,\m}(\b^\n)\neq 0\mod  \mathfrak p\}\\
            &= \{\y \in (\F_p^\times)^r \mid f_\m(\y,X)\in \F_p[X] \textrm{ has a root in $\F_p$ and $g_{d,\m}(\y)\neq 0$ in  $\F_p$}\}.
    \end{split}
    \]
    Write $\ca B = \{\y \in (\F_p)^r \mid f_\m(\y,X)\in \F_p[X] \textrm{ has a root in $\F_p$ and $g_{d,\m}(\y)\neq 0$ in  $\F_p$}\} $. Then, 
    \[
        {\rm Prob}(\n \in \ca A_{\mathfrak p})=\frac{\#\ca A_{\mathfrak{p}}}{(p-1)^r} = \frac{\# \ca B}{p^r} + O(p^{-1}),
    \]
    so it suffices to prove that $\frac{\# \ca B}{p^r} \leq 1-\frac1d + O_{f,\ell}(p^{-1/2})$. This is a classical bound; we follow 
    the arguments of \cite{Serre_JTHM}.
    Since $\mathfrak p$ splits completely in $L_{f_\m}$,  the Galois closure $F$ of $f_\m(\t, X)\in \F_p[\t,X]$ over $\F_p(\t)$ is regular over $\F_p$; let $H$ be the Galois group of $F/\F_q(\t)$ viewed as a permutation group via the action on the roots of $f_\m(\t,X)$. So $H$ is transitive, since $f_\m$ is irreducible. Let $\mathcal C$ be the set of $\sigma\in H$ having a fixed point. Then $\mathcal C$ is a union of conjugacy classes, and since $H$ is transitive, $|\mathcal C|/|H| \leq  1-1/d$ (\cite[Theorem~5]{Serre_JTHM}).
    
    By an explicit function field Chebotarev's density theorem (see e.g.\ \cite[Proposition 6.4.8]{friedjarden} or \cite[Theorem 3]{entin}) we conclude that  
    \[
     \frac{\# \ca B}{p^r}  = \frac{|\mathcal C|}{|H|} + O_{f,\ell}(p^{-1/2}) \leq  1 - \frac{1}{d} + O_{f,\ell}(p^{-1/2}),
    \]
    as needed.
\end{proof}

We show that the events $\ca A_{\mathfrak p}$ for distinct primes $\mathfrak p\in \ca P_{f,\ell}(\a)$ are almost independent under the assumption that the $p-1$ are `almost' coprime:
\begin{lemma}
\label{l_many_primes} 
    Let $\ell, x,z>0$. Let $\mathfrak p_1,\ldots, \mathfrak p_t\in \ca P_{f,\ell}(\a)$ be primes with respective norms  $p_1, \ldots, p_t$. Assume that $p_v\in (x,2x]$, $(p_v-1,p_u-1)\leq z$ for all $v\neq u$. Let $M$ be a common multiple of $p_v-1$, $v=1,\ldots, t$ and $\n$  a random variable taking the values of $(\Zint/M\Zint)^r$ uniformly. Then, 
    \[
        {\rm Prob}\Big(\n\in \bigcap_{v=1}^t \ca A_{{\mathfrak p}_v}\Big) \leq  (1-d^{-1})^t + O_{f,z,t,\ell}(x^{-1/2}),
    \]
    with $\ca A_{{\mathfrak p}_v}=\ca A_{{\mathfrak p}_v,M}$ as defined in \eqref{eq:Anp}.
\end{lemma}

\begin{proof}
    Given $u\geq 1$, let $M_u = [p_1-1,\ldots, p_{u-1}-1]$ (so that $M_1=1$) 
    and $\c\in \Zint^r$.
    Let
    \[
        P_{\c, u} := {\rm Prob}(\n\in \ca A_{{\mathfrak p}_u} \mid \n \equiv \c \mod M_{u}).
    \] 
    Let $g=(p_u-1,M_u)$; so $g\leq z^u$. Since the image of $\Zint/M_{u}(p_u-1) \Zint \to \Zint/M_{u} \Zint \times  \Zint/(p_u-1) \Zint $ is $ \Zint/M_{u} \Zint\times_{\Zint/g \Zint} \Zint/(p_u-1) \Zint$, we get that
    \[
        P_{\c, u} = {\rm Prob}(\n\in \ca A_{{\mathfrak p}_u} \mid \n \equiv \c \mod g).
    \]
    If $\n\equiv \c\mod g$, we may write $\n = \c + \x g$ with $\x$ uniform in $(\Zint/M\Zint)^r$. Then $a_i^{n_i} = a_i^{c_i} (a_i^{x_i})^g$. 
    Let 
    \[
        \tilde{f}(\t,X) = f(a_1^{c_1}t_1^{g},\ldots, a_r^{c_r}t_r^g,X) =f_{g}(\a^\c \t,X),
    \]
    where $\a^\c\t :=(a_1^{c_1}t_1,\ldots, a_r^{c_r}t_r)$ and $f_g(\t,X) = f(\t^g,X)$. Then $f(\a^\n,X) = \tilde{f}(\a^{\x},X)$. 
    
    The splitting fields of $f_g$ and $\tilde{f}$ over $K(\t)$ are isomorphic over $K$. Thus $\mathfrak p_u$ satisfies condition (ii) in the definition of $\ca P_{\tilde{f}}$, see  \eqref{eq_p}. Also, the isomorphism of the splitting fields implies that $L_{\tilde{f}} = L_{f_{g}}$, and by  Lemma~\ref{l_field_scalars}, $L_{f_{g}}=L_f$, so that $L_{\tilde{f}}=L_{f}$. Thus, $\mathfrak p_u$ satisfies also condition (iii), and hence $\mathfrak p_u\in \ca P_{\tilde{f},\ell}$. 
    Applying Lemma~\ref{l_one_prime} to $\tilde{f}$, and recalling that $\x$ is uniform and that $\tilde{f}$ depends only on $f$, $z$, and $t$, we get
    \[
    \begin{split}
        P_{\c,u} &= 
            {\rm Prob}( 
                \n\in \ca A_{\mathfrak{p}_u} \mid \n\equiv \c\mod g
            ) 
            = {\rm Prob}(\c+\x g \in \ca A_{\mathfrak p_u})
            \\
        &= {\rm Prob}( f(\a^{\c+\x g},X) 
        \mbox{ 
            has a root modulo $\mathfrak p$ and $g_d(\a^{\c+\x g})\neq 0$ in $\F_{p_u}$\}
            }
            )
        \\
        &=
        {\rm Prob}( \tilde{f}(\a^\x ,X) 
        \mbox{ 
            has a root modulo $\mathfrak p$ and $\tilde{g}_d(\a^\x):=g_d(\a^{\c+\x g})\neq 0$ in $\F_{p_u}$\}}
            )\\   
            &\leq 1-d^{-1} + O_{\tilde{f},\ell}(p_u^{-1/2}) = 1-d^{-1} + O_{f,z,t,\ell}(x^{-1/2}).
    \end{split}
    \]         
    By the law of total probability,
    \[
    \begin{split}
        P_{u} &:= {\rm Prob}\Big(\ca A_{{\mathfrak p}_u} \Big| \bigcap_{v=1}^{u-1} {\ca A_{{\mathfrak p}_v}}\Big) \\
        &= \sum_{\substack{\c \mod M_u \\ 
        \c\in \bigcap_{v=1}^{u-1} \ca A_{{\mathfrak p}_v}}} {\rm Prob} (\n\in A_{{\mathfrak p}_u} \mid \n\equiv \c\mod M_u) {\rm Prob} (\n\equiv \c \mod M_u \Big| \bigcap_{v=1}^{u-1} {\ca A_{{\mathfrak p}_v}}) \\ 
        &\leq 1-d^{-1} + O_{f,z,t,\ell} (x^{-1/2}).
    \end{split}
    \]
    Therefore, 
    \[
        {\rm Prob}\Big(\bigcap_{v=1}^t\ca A_{{\mathfrak p}_v}\Big) = \prod_{u=1}^t P_{u} \leq (1-d^{-1})^t + O_{f,z,t,\ell}(x^{-1/2}),
    \]
    as needed.
\end{proof}

\subsection{Thin set of type I}
Let $C$ be a Zariski closed proper subvariety of $\G_m^r$, let $\a\in \ca O_K$ with $a_i$ nonzero and not roots of unity. For a sufficiently large prime $\mathfrak p$ of $\mathcal{O}_K$ with $\mathrm N_{K/\Q}(\mathfrak p) = p$ and $a_i\not\in \mathfrak p$, we set 
\[
Z_{C,N,\mathfrak p}:=\{\n \in (\Zint \cap [-N,N])^r \mid \a^\n\in C(\F_p)\}.
\]

\begin{lemma}
    \label{l:zariski_closed}
    Let $C$ be a Zariski closed proper subvariety of $\G_m^r$, let $\a\in \ca O_K$ with $a_i$ nonzero and not roots of unity, let $\mathfrak p$ be a prime of $\ca O_K$ such that $\mathrm N_{K/\Q}(\mathfrak p)=p$ is prime, and let $\ell >0$. Assume that for each $i$, the order of $a_i\mod \mathfrak p$ is at least $(p-1)/\ell$. If $p$ is sufficiently large depending on $C$, then 
	\[
	\frac{\#Z_{C,N,\mathfrak p}}{(2N+1)^r} = O_{\deg C,\ell}(p^{-1}) + O(pN^{-1}).
	\]
\end{lemma}

\begin{proof}
    By the assumption that $p$ is sufficiently large, we may assume that the reduction of $C$ modulo $\mathfrak p$ is a proper Zariski-closed subvariety of $G = (\G_m^r)_{\F_p}$.

    Similarly to the 2nd-4th paragraphs of the proof of Lemma~\ref{l_one_prime}, we may replace $C$ by $C_\m = C\times_{[\m]} G$, where $(p-1)/m_i$ is the order of $a_i$ with\ $m_i \leq \ell$, and then
    \[
        \frac{\# Z_{C,N,\mathfrak p}}{(2N+1)^r} = \frac{\#\{ \y\in (\F_p^\times)^r \mid \y\in C_\m(\F_p)\}}{(p-1)^r} + O(\frac pN).
    \]
    By the Lang-Weil estimates, and since $\deg C_{\m}\ll_{\ell} \deg C$, we have $\#\{ \y\in (\F_p^\times)^r \mid \y\in C(\F_p)\}\ll_{\deg C,\ell} p^{r-1}$, so the result follows.
    \end{proof}

\section{Proof of Theorem~\ref{t:main_no_roots}
}
    Let $K$ be a number field with ring of integers  $\ca O_K$, let $\a=(a_1,\ldots, a_r)\in \ca O_K^r$ be such that each $a_i$ is nonzero and not a root of unity. Let $f\in \ca O_K[\t,X]$ be a polynomial satisfying (PB). 
    Let $\ell$ be sufficiently large, so that $\underline{\delta}(\ca P_{f,\ell}(\a)) >0$  (Lemma~\ref{l:large_ell}).
    
    We need to prove that $\displaystyle\lim_{N\to \infty}{\rm Prob}(\n \in \ca N^0_K(f,\a;N)) =1$ (under GRH), where $\n$ is a  random variable taking the values of $([-N,N]\cap \Zint)^r$ uniformly at random and $\ca N^0_K$ is as defined in \eqref{def:noroots}. 

    We let $t, z, x$ be three parameters depending on $N$ satisfying the constraints \eqref{cons_0}, \eqref{cons_1}, \eqref{cons_2}, \eqref{cons_3}, and \eqref{cons_4}, below.
    The first constraint is 
    \begin{equation}
        \label{cons_0}
        \lim_{N\to \infty}t=\lim_{N\to \infty}z=\lim_{N\to \infty}x=\infty.
    \end{equation}
    By Lemma~\ref{l:graph_theory} applied to $\ca P=\mathcal{P}_{f,\ell}(\a)$, there exists $c>0$ depending only on $\ell, f,\a$ such that if
    \begin{equation}\label{cons_1}
        ct\leq z\leq \log x,
    \end{equation}
    then there exist $\mathfrak{p}_1\ldots, \mathfrak{p}_t \in \ca P_{f,\ell}(\a)$ of respective norms $p_1, \ldots, p_t\in  (x,2x]$ such that $(p_i-1,p_j-1)\le z$ for all $i\neq j$. 

    Let $C=\{ g_d(\t)=0\}$ be the zero set of $g_d(\t)$. By Lemma~\ref{l:zariski_closed},  
    \begin{equation}
    \label{eq:prob_gdzero}
        {\rm Prob}\Big(\n \in \bigcup_{i=1}^ t  Z_{C,N,\mathfrak p_i}\Big) = O(tx^{-1} + txN^{-1}) \to 0,
    \end{equation}
    as $N\to \infty$, by \eqref{cons_1} and provided
    \begin{equation}\label{cons_2}
        tx=o(N).
    \end{equation}

    Let $M = \prod_{i=1}^t (p_i-1)< (2x)^t$ and let $\m$ be a uniform random variable on $(\Zint/M\Zint)^r$. Then, the total variation distance between the distribution of $\n\mod M$  from the uniform distribution modulo $M$ is $O((2x)^tN^{-1})$. 
    So if
    \begin{equation}
        \label{cons_3}
        (2x)^t=o(N),
    \end{equation}
    Lemma~\ref{l_many_primes} implies that there exists  $\alpha_{z,t}>0$ depending only on $z,t,f,\a,\ell$ and not on $x$ and $N$ such that
    \begin{equation}\label{eq:ApN}
        {\rm Prob} \Big(\n\mod M \in \bigcap_{i=1}^t\ca A_{\mathfrak p_i,M} \Big) ={\rm Prob} \Big(\m\in \bigcap_{i=1}^t\ca A_{\mathfrak p_i,M} \Big) +o(1)\leq (1-d^{-1})^t + \alpha_{z,t}x^{-1/2}+o(1) \to 0,
    \end{equation}
    provided $z,t$ tend to infinity sufficiently slow so that 
    \begin{equation}
        \label{cons_4}
        \lim_{x\to \infty}\alpha_{z,t}x^{-1/2} = 0.
    \end{equation}

    Now, if $\n\not\in \ca N^0_K(f,\a;N)$ and $\n \not\in \bigcup_{i=1}^ t  Z_{C,N,\mathfrak p_i}$, then $f(\a^\n,X)$ has a root in $K$ and $g_d(\a^\n)\neq 0 \mod \mathfrak p_i$, so $f(\a^n,X)$ has a root modulo $\mathfrak p_i$ for all $i$, i.e.\ $\n\mod  M\in \bigcap_{i=1}^t \ca A_{\mathfrak p_i,M}$. Hence, by \eqref{eq:prob_gdzero} and \eqref{eq:ApN}
     \[
        {\rm  Prob}(\n \not \in \ca N^0_K(f,\a;N)) \leq {\rm Prob}\Big(\n \in \bigcup_{i=1}^ t  Z_{C,N,\mathfrak p_i}\Big) + {\rm Prob}(\n\mod  M\in \bigcap_{i=1}^t \ca A_{\mathfrak p_i,M}) \to 0 
    \]
    as $N\to \infty$. This finishes the proof as it is obvious we can choose $t,z,x$ satisfying \eqref{cons_0}, \eqref{cons_1}, \eqref{cons_2}, \eqref{cons_3}, and \eqref{cons_4}.
\qed

\bibliography{references}
\bibliographystyle{alpha}

\end{document}